\documentclass[10pt]{amsart}
\usepackage{amsmath,amssymb,amscd,amsthm,amsfonts,amstext,amsbsy,mathrsfs,hyperref,upgreek,mathtools,stmaryrd,enumitem,bbm, dsfont, turnstile, graphicx, color, verbatim}

\usepackage%[scale=0.682]
[a4paper, margin=1.2in]{geometry}

\hypersetup{colorlinks=true}

\newcommand\V{\mathsf{V}}

\newcommand{\RR}{\mathbb{R}}

\newcommand{\PP}{\mathbb{P}}
\newcommand{\QQ}{\mathbb{Q}}

\newcommand{\sa}{\mathbb{S}}

\newcommand{\lin}[3]{{\text{span}}\{#1_{#2}: #2<#3\}}
\newcommand{\Span}[1]{{\text{span}}(#1)}

\newcommand{\seq}[2]{\langle #1 \mid #2 \rangle}

\newcommand{\bij}{\longleftrightarrow}

\newcommand{\w}{\omega}

\newcommand{\centm}[1]{\begin{center}#1\end{center}}
\newcommand{\supp}[1]{{{\rm{supp}}(#1)}}

\newcommand{\pow}{\mathscr{P}}

\newcommand{\Spl}{\operatorname{split}}

\newcommand{\stem}{\operatorname{stem}}
\newcommand{\Lev}[2]{\operatorname{Lev}_{#1}{(#2)}}

\newcommand{\res}[2]{#1\!\upharpoonright\!{#2}}

\newcommand{\lh}{\operatorname{lh}}

\newcommand{\one}{\mathds{1}}
\newcommand{\card}[1]{|#1|}%{\overline{\overline{#1}}}

\newcommand{\ZFC}{{\sf ZFC}}
\newcommand{\ZF}{{\sf ZF}}
\newcommand{\AC}{{\sf AC}}
\newcommand{\DC}{{\sf DC}}

\newcommand{\OR}{{\sf OR}}
\newcommand{\CH}{{\sf CH}}

\newcommand{\VV}{{V}}

\newcommand{\LL}{L}

\newtheorem{theorem}{Theorem}[section]
\newtheorem{lemma}[theorem]{Lemma}
\newtheorem{corollary}[theorem]{Corollary}
\newtheorem{proposition}[theorem]{Proposition}

\newtheorem{claim}{Claim}
\newtheorem*{claim*}{Claim}
\newtheorem*{subclaim*}{Subclaim}

\theoremstyle{definition}
\newtheorem{definition}[theorem]{Definition}
\newtheorem*{notation}{Notation}

\theoremstyle{remark}
\newtheorem{remark}[theorem]{Remark}

\newenvironment{enumerate-(1s)}{\begin{enumerate}[label={\upshape (\arabic{*}*)}]}{\end{enumerate}}
\newenvironment{enumerate-(is)}{\begin{enumerate}[label={\upshape (\roman{*})*}]}{\end{enumerate}}
\newenvironment{enumerate-(iss)}{\begin{enumerate}[label={\upshape (\roman{*}**)}]}{\end{enumerate}}

\newenvironment{enumerate-(a)}{\begin{enumerate}[label={\upshape (\alph*)}]}{\end{enumerate}}

\newenvironment{enumerate-(5a)}{\begin{enumerate}[label={\upshape (6-\alph*)}]}{\end{enumerate}}

\newenvironment{enumerate-(a)-r}{\begin{enumerate}[label={\upshape (\alph*)}, resume]}{\end{enumerate}}

\newenvironment{enumerate-(A)}{\begin{enumerate}[label={\upshape (\Alph*)}]}{\end{enumerate}}

\newenvironment{enumerate-(A)-r}{\begin{enumerate}[label={\upshape (\Alph*)}, resume]}{\end{enumerate}}

\newenvironment{enumerate-(i)}{\begin{enumerate}[label={\upshape (\roman*)}]}{\end{enumerate}}

\newenvironment{enumerate-(i)-r}{\begin{enumerate}[label={\upshape (\roman*)},resume]}{\end{enumerate}}

\newenvironment{enumerate-(I)}{\begin{enumerate}[label={\upshape (\Roman*)},]}{\end{enumerate}}

\newenvironment{enumerate-(I)-r}{\begin{enumerate}[label={\upshape (\Roman*)},resume]}{\end{enumerate}}

\newenvironment{enumerate-(1)}{\begin{enumerate}[label={\upshape (\arabic*)}]}{\end{enumerate}}

\newenvironment{enumerate-(1)-r}{\begin{enumerate}[label={\upshape (\arabic*)},resume]}{\end{enumerate}}

\definecolor{teal2}{rgb}{0.036, 0.512, 0.512}

\usepackage{hyperref}
\hypersetup{colorlinks=true}
\hypersetup{citecolor=teal2,linkcolor=teal2}

\begin{document}
\nocite{*} %%las referencias en el documento aparecen a�n cuando no son citadas

\bibliographystyle{acm}

%\subjclass[2010]{03E30, 03E40, 03E70} 

%\keywords{Class forcing, Separation, Replacement} 

\author{J\"org Brendle}
\address{J\"org Brendle, The Graduate School of Science and Technology, Kobe University, 
Rokko-dai 1-1, Nada Kobe 657-8501, Japan}
\email{brendle@kobe-u.ac.jp}
\author{Fabiana Castiblanco}\thanks{The second and third authors gratefully
acknowledge support from the SFB 878 ``\emph{Groups, Geometry \& Actions},'' a grant by the DFG (Deutsche Forschungsgemeinschaft). The third author also thanks Vladimir Kanovei for his valuable comments on earlier drafts of this paper.}

\address{Fabiana Castiblanco, Institut f\"ur Matematische Logik und Grundlagenforschung, Universit\"at M\"unster, 
Einsteinstra{\ss}e 62, 48149 M\"unster, Germany}
\email{fabi.cast@wwu.de}

\author{Ralf Schindler}
\address{Ralf Schindler, Institut f\"ur Matematische Logik und Grundlagenforschung, Universit\"at M\"unster, 
Einsteinstra{\ss}e 62, 48149 M\"unster, Germany}
\email{rds@wwu.de}

\author{Liuzhen Wu}
\address{Liuzhen Wu, Institute of Mathematics, Chinese Academy of Sciences, East Zhong Guan Cun Road N0. 55, Beijing 100190, P.R. of China}
\email{lzwu@math.ac.cn}

\author{Liang Yu}\thanks{The fifth author was partially supported by NSF of China No.No. 11671196.}
\address{Liang Yu, Institute of Mathematical Sciences, Nanjing University, Nanjing Jiangsu Province 210093, P.R. of China}
\email{yuliang.nju@gmail.com}

%\title{}
%\date{\today}

\title{A model with everything except for a well-ordering of the reals}

\begin{abstract} 
We construct a model of $\ZF+\DC$ containing a Luzin set, a Sierpi\'{n}ski set, as well as a Burstin basis but in which there is no a well ordering of the continuum.   \end{abstract} 

\maketitle

\setcounter{tocdepth}{1}
%\tableofcontents 

\section{Introduction} 

In this paper we study subsets of the real line $\RR$ with specific properties  whose classic constructions were performed by assuming various forms of the Axiom of Choice (AC).\:  The first \emph{pathological}\: set was constructed by F. Bernstein in 1908 (cf.\ \cite{bern}); he constructed a set $B\subset \mathbb{R}$ of cardinality the continuum  such that neither $B$ nor $\mathbb{R}\setminus B$ contains a perfect subset of reals.  Such a set can be obtained by assuming the existence of a well-ordering of $\mathbb{R}$.   Later in 1914, Luzin constructed an uncountable set  $\Lambda\subset \RR$ having countable intersection with every meager set (cf.\ \cite{luz}).  His construction required the continuum hypothesis (\textsf{CH}, in the strong form according to which $\RR$ may be well-ordered
in order type $\omega_1$).  In 1924, Sierpi\'nski developed a similar construction to the one given by Luzin; under the assumption of the same form of \textsf{CH}, he constructed an uncountable set $S\subset \RR$ having countable intersection with every measure zero set (cf.\ \cite{sier}).   

However $\textsf{CH}$ is not a necessary assumption for the existence of Luzin and Sierpi\'{n}ski sets (see  \cite{misp}). 
Moreover a Luzin set may exist in a model in which the set of reals is not well-ordered.  In fact, D. Pincus and K. Prikry \cite{pinc}  proved that in the Cohen-Halpern-L\'evy model $H$, a model in which the reals cannot be well-ordered (in fact, in $H$ there is an uncountable set of reals with no countable subset),  there is a Luzin set as well as a Vitali set.    Additionally, Pincus and Prikry asked whether a Hamel basis, i.e., a basis for $\RR$ construed as a vector space over the field of rational numbers $\mathbb{Q}$,  exists in $H$ or, in general, if the existence of a Hamel basis is compatible with the non-existence of a well-ordering of the reals.  Recently, M. Beriashvili, R. Schindler, L. Wu and L. Yu (cf. \cite{hamel2}) answered this question in the affirmative, by showing that in $H$ there is a Hamel basis and, furthermore, in $H$ there is also a Bernstein set (see \cite[Theorems 1.7 and 2.1]{hamel2}).  Thus the model $H$ has many pathological sets of reals, but in $H$ the continuum cannot be well ordered.   There is no Sierpi\'{n}ski set in $H$, though (see \cite[Lemma 1.6]{hamel2}). 

Let us informally refer to a model $M$ as a ``Solovay model'' iff 
$M$ is obtained via a symmetric collapse over a model in which 
what is to become $\omega_1^M$ is either inaccessible or a limit
of large cardinals (e.g., Woodin cardinals).
The paper \cite{carlos_stevo} shows that if $U$ is a selective ultrafilter
on $\omega$ which was added by forcing over a Solovay model $M$, then 
$M[U]$ satisfies the Open Coloring Axiom (see \cite[p.\ 247]{carlos_stevo}),
hence $M[U]$ inherits from $M$ the property that every uncountable set
of reals that a perfect subset and in particular 
$M[U]$ does not contain a well--ordering of the 
reals, see \cite[Theorem 5.1]{carlos_stevo}. 

The paper \cite{larson_zapletal}
further explores this topic and studies which consequences of having a well--ordering
of ${\mathbb R}$ remain false when adding certain ultrafilters on $\omega$
over a Solovay model or when adding a Vitali set. Also,
\cite{larson_zapletal} produces a model of {\sf ZF} plus 
{\sf DC} plus ``there is a Hamel basis'' plus ``there is no well--ordering of the
reals.'' The verification in \cite{larson_zapletal}
that the extension of the Solovay model via forcing with countable
linearly independent sets of reals (called ${\mathbb Q}_H$ in the current paper, see 
Definition \ref{defn_hamel_basis_forcing_2} below) doesn't
have a well--ordering of its reals uses 
large cardinals, specifically 
Woodin's stationary 
tower forcing. The forcing ${\mathbb Q}_H$ used by \cite{larson_zapletal} does not work in the absence of 
large cardinals, though, see Corollary \ref{key_corollary} below. 

The current paper improves the result obtained in \cite{hamel2} by showing that there is a model $W$ of $\ZF+\DC$ such that in $W$ the reals cannot be well-ordered and $W$ contains Luzin as well as Sierpi\'{n}ski sets and also a Burstin basis, i.e., a set which is simultaneously a Hamel basis and a Bernstein set.  Notice that from the existence of a Hamel basis one can derive that in $W$ there is also a Vitali set (see \cite[Lemma 1.1]{hamel2}).   

\section{Basic definitions and results}

\subsection{Pathological sets within \ZFC}
\begin{definition}\label{defs}Let $A\subseteq\RR$ uncountable.  We say that $A$ is

\begin{enumerate-(i)}\item a \index{special sets!Vitali}\emph{Vitali set} if $A$ is the range of a selector for the equivalence relation $\sim_\QQ$ defined over $\RR\times \RR$ by $x\sim_\QQ y\iff x-y\in\QQ$;
\item a \emph{Sierpi\'{n}ski set}\index{special sets!Sierpi\'{n}ski} if for every $N\in \mathcal{N}$ -the ideal of null-sets with respect to Lebesgue measure over $\RR$- we have $\card{A\cap N}\leq\w$; 
\item a \emph{Luzin set}\index{special sets!Luzin} if for every $M\in\mathcal{M}$ -the ideal of the Borel meager sets- we have $\card{A\cap M}\leq\w$;
\item a \emph{Berstein set}\index{special sets!Bernstein} if for every perfect set $P\subseteq \RR$ we have $A\cap P\neq \varnothing \neq (\RR\smallsetminus A)\cap P$;
\item a \emph{Hamel basis}\index{special sets!Hamel basis} if $A$ is a maximal linearly independent subset of $\RR$ when we consider it as a vector space over $\QQ$.
\item a \emph{Burstin basis}\index{special sets!Burstin} if $A$ is a Hamel basis which has nonempty intersection with every perfect set.
\end{enumerate-(i)}
\end{definition}

The existence of a Hamel basis in a model of $\ZF+\DC$ implies the existence of nonmeasurable sets and the existence of sets without the Baire property. In particular, we have the next result connecting  Hamel bases and Vitali sets.  For a proof, see \cite[Lemma 1.1]{hamel2}.

\begin{lemma}\label{vi} {\bf (Folklore)} Suppose $V\models \ZF$ and suppose that a Hamel basis $H$ exists.  Then there is a Vitali set.\end{lemma}

\begin{lemma} {\bf (Luzin, 1914, and Sierpi\'{n}ski,1924)} Assume $\VV$ is a model of $\ZFC+ \CH$.  Then, there are $\Lambda$ and $S$ in $\VV$ such that $\Lambda$ is a Lusin set and $S$ is a Sierpi\`{n}ski set. \end{lemma}

\begin{proof}Let $\{N_i: i<\w_1\}$ be an enumeration of all $G_\delta$ null sets.  Recursively define $\langle x_i:i<\w_1\rangle$ such that $x_i\notin \bigcup\{N_j:j<i\}\cup\{x_j:j<i\}$.  Then, $S=\{x_i:i<\w_1\}$ is a Sierpi\'{n}ski set.

The same procedure gives us a Luzin set, starting out with an enumeration $\{M_i:i<\w_1\}$ of all $F_\sigma$-meager sets.\end{proof}

\begin{remark}As we may write $\RR=N\cup M$ where $N$ is null and $M$ is meager, no set can be both a Sierpi\'{n}ski set as well as a Lusin set.\end{remark}

The construction of a Bernstein set in $V$ is based on the enumeration of all perfect subsets of $\RR$.   We omit the proof and instead present below the construction of
a Burstin basis in $V$ under $\AC$ (see Theorem \ref{burstin}).  
 \begin{proposition} {\bf (Folklore)} Every Burstin basis is a Bernstein set.\end{proposition}
 
 \begin{proof}Suppose $B\subseteq \RR$ is a Burstin basis such that $P\subseteq B$ for some perfect $P\subseteq\RR$.  As $B$ is linearly independent, the set $2P=\{2p: p\in P\}$ has empty intersection with $B$.  On the other hand, $2P$ is a perfect set, so $2P\cap B\neq\varnothing$, which gives a contradiction.   It follows that $B$ is totally imperfect, so $(\RR\smallsetminus B)\cap P\neq \varnothing$ as well, i.e., $B$ is a Bernstein set.    \end{proof}

It is easy to construct a Hamel basis $H$ such that $H \cap P = \emptyset$ for some
perfect set $P$; no such $H$ can then be a Burstin basis. It is also not hard to construct
a Hamel basis $H$ which contains a perfect set (see e.g.\ \cite[Example 1, p.\ 477f.]{jones}); no such $H$ can be a Burstin basis either.
  
\begin{theorem}\label{burstin} {\bf (Burstin, 1916)} Assume $\VV\models \ZFC$.  Then there is a Burstin basis $B$.\end{theorem}
 
 \begin{proof}Suppose $\{P_i:i\leq2^{\aleph_0}\}$ is a enumeration of all perfect subsets of $\RR$.  By transfinite recursion we are going to define a set $\{b_\alpha:\alpha<2^{\aleph_0}\}\subseteq\RR$ such that 
 
 \begin{enumerate-(i)}\item $b_\alpha\in P_\alpha$ for every $\alpha<2^{\aleph_0}$
 \item for every $\beta<2^{\aleph_0}$, the set $\{b_\alpha:\alpha<\beta\}$ is linearly independent
 \end{enumerate-(i)} 
 
 Suppose that $\beta<2^{\aleph_0}$ and we already have defined the collection $\{b_\alpha:\alpha<\beta\}$ satisfying (i) and (ii) above.  
 
 Consider the set $\lin{b}{\alpha}{\beta}$.  Note that $$\card{\lin{b}{\alpha}{\beta}}\leq \card{\beta}+\w<2^{\aleph_0}$$
 
Thus, $P_\beta\smallsetminus \lin{b}{\alpha}{\beta}\neq\varnothing$ and we may pick an element $b_\beta$ from this set.  

According to this procedure, we have constructed a linearly independent family $\{b_\alpha:\alpha<2^{\aleph_0}\}$ satisfying (i).  We can extend this family to a maximal one, call it $B$, and in this way, $B$ will be a Hamel basis over $\RR$.  

By construction, $B$ intersects every perfect subset of $\RR$, so $B$ is in fact a Burstin basis. 
 \end{proof}

\subsection{The Marczewski ideal and new generic reals}

Before the appearance of the forcing technique, in 1935 E. Marczewski introduced the $\sigma$-ideal $s^0$.  This ideal is related to Sacks forcing in much the same way that Cohen forcing is related with the ideal of meager subsets of $\RR$ and Random forcing is related with the ideal of Lebesgue null subsets of $\RR$.

\begin{definition} {\bf (Marczewski, 1935)} A set $X\subseteq{}^\w2$\index{Marczewski ideal $s_0$} is in $s^0$ if and only if for every perfect tree $T\subseteq{}^{<\w}2$, there is a perfect subtree $S\subseteq T$ with $[S]\cap X=\varnothing$. \end{definition}
It is easy to see that $s^0$ is an ideal which does not contain any perfect set.    Furthermore, any subset $X$ of the reals with $|X|<2^{\aleph_0}$ is in the Marczewski ideal, as well as every universal measure zero set and every perfectly meager set\footnote{A set $N^*\subseteq {}^\w2$ has universal measure zero if for every measure $\mu$ defined on the Borel sets of ${}^\w2$, there is $B$ a $\mu$-null Borel set such that $N^*\subseteq B$.  Analogously, we say that $M^*\subseteq {}^\w2$ is perfectly meager if for every perfect tree $T\subseteq{}^{<\w}2$, the set $M^*\cap [T]$ is meager relative to the topology of [T].}.  However, $s^0$ contains sets of size continuum (cf.\ \cite[Theorem 5.10]{misp}).  Moreover, by a ``fusion'' argument we can see that $s^0$ is a $\sigma$-ideal, i.e. closed under countable unions.
\begin{remark}We say that $X\subseteq {}^\w2$ is $s$-measurable\index{s@$s$-measurability} if for each  $T\in\sa$ there is $S\leq T$ such that either $[S]\cap X=\varnothing$ or $[S]\subseteq X$.  Note that the algebra of the $s$-measurable sets modulo the ideal $s^0$ corresponds, in fact, to Sacks forcing.   \end{remark}

\begin{definition}\index{countable covering for reals}Suppose that $M\subseteq N$ are models of $\ZFC$.  We say that the pair $(M,N)$ satisfies \emph{countable covering for reals} if for every $A\subseteq {}^\w2^M$, $A \in N$, such that $A$ is countable in $N$, there is a set $B\subseteq{}^\w2^M$, $B \in M$, such that $A\subseteq B$ and $B$ is countable in $M$. \end{definition}
In the 1960's, K. Prikry asked whether the existence of a non constructible real implies the existence of a perfect set of non constructible reals (cf. \cite{msurr}).   In order to find a solution to Prikry's problem,  Marcia J. Groszek and Theodore A. Slaman  have shown the following result in \cite[Theorem 2.4]{groszek}\footnote{See also \cite[Theorem 3]{vewo}}:
\begin{theorem}\label{groo}Suppose that $M\subseteq N$ are models of\: $\ZFC$ such that $(M,N)$ satisfies countable covering for reals.  Then every perfect set  $P\subseteq{}^\w2^N$ in $N$ has an element which is not in $M$. \end{theorem}
In \cite[\S1]{groszek}, the authors state without proof that the conclusion in \ref{groo} can be strengthened to: for every perfect set $P\subseteq{}^\w2^N$ in $N$
there is a perfect set $P'\subseteq P$ in $N$ such that $P'\cap M=\varnothing$, which is equivalent to saying that ${}^\w2^M\in s_0^N$ ($s_0^N$ being $s_0$ of $N$).
In what follows we present a proof of this strengthened version of \cite[Theorem 2.4]{groszek}.  

\begin{theorem} {\bf (Groszek-Slaman)} \label{gro}\index{Groszek-Slaman Theorem}  Let $W\subseteq \VV$ be an inner model such that $W\models \CH$. If \: ${}^\w2^\VV\smallsetminus {}^\w2^W\neq \varnothing$ holds, we have\centm{$V\models {}^\w2^W\in s^0$}  
\end{theorem}

\begin{proof} We may assume that $\w_1^W=\w_1^\VV$, as otherwise $W$ has only countably many reals and the result is trivial.\:  
\begin{claim}\label{countablycovering}
The pair $(W,V)$ satisfies countable covering for reals. 
\end{claim}
\begin{proof}Suppose that $A\in \VV$ is a countable set such that $A\subseteq {}^\w2^W$.   Since $\w_1^W=\w_1^\VV$ and $W\models \CH$ we can take a well-ordering of ${}^\w2^W$ in $W$ of length $\w_1$.   Then, there is some $\alpha<\w_1^W$ such that $A\subseteq \{a_i:i<\alpha\}$ where $\{a_i:i<\w_1^M\}$ is an enumeration of ${}^\w2^W$ according with the prefixed well-ordering.   Therefore, $B=\{a_i:i<\alpha\}\in W$ is countable in $W$ and covers $A$.\end{proof}

Let us fix a perfect set $P\subseteq {}^\w2$ in $\VV$. We aim to find  a perfect subset $\bar{P}\subseteq P$ such that $\bar{P}\cap {}^\w2^W=\varnothing$, or, equivalently $\bar{P}\subseteq \VV\smallsetminus W$.\:  Let  $T\subseteq {}^{<\w}2$ be a perfect tree such that $P=[T]$.  We call $x\in[T]$ \emph{eventually trivial}\index{tree!eventually trivial branch of a} if and only if there is some finite $s\subsetneq x$ such that  $x$ is the leftmost or the right most branch of $T_s$.    We consider two cases:

\noindent \textbf{Case 1.}\: Suppose that there is some $s\in T$ such that if $x\in[T_s]$ is not eventually trivial then $x\in \VV\smallsetminus W$.\:  In this situation we have that $[T_s]\cap W$ is a subset of all eventually trivial elements of $[T_s]$; since the later set is countable there is some perfect set $\bar{P}\subseteq[T_s]$ consisting only of elements of $\VV\smallsetminus W$.   But then $\bar{P}\subseteq [T_s]\subseteq P$.

\noindent \textbf{Case 2.}  Now suppose that for all $s\in T$, there is some $x\in[T_s]\cap W$ which is not eventually trivial.\:   For each $s\in T$, pick $x_s\in[T_s]\cap W$ not eventually trivial.  Let $\vec{g}=\seq{g_n}{n<\w}\in W$  be a sequence of elements of ${}^\w2\cap W$ such that for all $s\in T$, there is some $n<\w$ such that $x_s=g_n$. $\vec{g}$  exists by \ref{countablycovering}.\:    We shall also assume that $g_0=x_{s_0}$ for some $s_0\in T$. 

\noindent  First, we prove $P\cap(V\smallsetminus W)\neq\varnothing$. \: Fix $r\in ({}^\w2\cap V)\smallsetminus W$ and construct $x,y\in{}^{\w}2$ and subsequences $\vec{g}^{\: x}, \vec{g}^{\: y}$ of $\vec{g}$ such that $x,y\in[T]$ and
\begin{enumerate-(1s)}\label{1s}\item $r\leq_T x,\: \vec{g}^{\: x}$,\: and  
\item $\vec{g}^{\: x},\: {\vec{g}}^{\: y}\leq_T x,\: y,\: \vec{g}$\end{enumerate-(1s)}

\noindent Thus, we have that $r\leq_T x, y, \vec{g}$.  But then, $x\in \VV\smallsetminus W$ or $y\in V\smallsetminus W$ and hence $P$ will have a member in $\VV\smallsetminus W$.  In a second round we shall actually produce a perfect $\bar{P}\subseteq P$, $\bar{P}\subseteq V\smallsetminus W$.  

\noindent We shall produce recursively strict initial segments of $x$ given by $\vec{g}^{\: x} =\seq{g_n^x}{n<\w}, \: y$ and $\vec{g}^{\:y}=\seq{g_n^y}{n<\w}$ as follows.

%In a second round we will actually produce a perfect $\bar{P}\subseteq P$, $\bar{P}\subseteq \VV\smallsetminus W$.  

We start with $g_0^x=g_0=g_0^y$.  We shall maintain inductively that $m=m(n), k=k(n)$ are such that $k\geq m\geq n$.  Suppose we are given $\res{x}{m(n)}$, $\vec{g_n}^{\: x}$, $\res{y}{k(n)}, \vec{g_n}^y$ such that

\begin{enumerate-(a)}\item $\res{x}{m(n)}\subsetneq \vec{g}_n^{\: x}$,
\item $\vec{g}_n^{\:x}=x_s$\: for some $s\in T$,
\item $\res{y}{k(n)}\subsetneq \vec{g}_n^{\: y}$, and
\item $\vec{g}_n^{\: y}= x_{s'}$\:  for some $s'\in T$.
\end{enumerate-(a)}

\noindent For $n=0$, we may just let $m=0=k$ and then (a) through (d) will be satisfied.   

\noindent Now say $\vec{g}_n^{\: y}=\vec{g}_j$.  Pick $m'>m(n), k(n)$ such that $\res{\vec{g}_l}{m'}\neq \res{\vec{g}_j}{m'}$ for all $l<j$.  By item (b), we may also assume that $\res{\vec{g}_n^{\: x}}{m'}$ is a splitting node in $T$ and $\vec{g}_n^{\: x}(m')\neq r(n)$.

 %Then, $j$ is least such that $\res{y}{m'}= \res{\vec{g}_j}{m'}$ for $m'$ maximal satisfying $\res{x}{m'}=\res{\vec{g}_n^x}{m'}$.  
\noindent Then set \centm{$\res{x}{m'+1}={\res{\vec{g}_n^{\: x}}{m'}}^\frown r(n)$}
and pick $\vec{g}_{n+1}^{\: x}$ such that for $s'':=\res{x}{m'+1}\in T$ we have $\vec{g}_{n+1}^{\: x}= x_{s''}$ and $\res{x}{m'+1}\subsetneq x_{s''}$.

\noindent Say $\vec{g}_{n+1}^{\: x}=\vec{g}_i$.  Pick $k'>m'+1$ such that $\res{\vec{g}_l}{k'}\neq \res{\vec{g}_i^{\: x}}{k'}$ for all $k<i$.  By (d), we may also assume that $\res{\vec{g}_n^{\: y}}{k'}$ is a splitting node. 

\noindent Then, set 
\centm{$\res{y}{k'+1}=\res{\vec{g}_n^{\: y}}{k'}^\frown (1-\res{\vec{g}_n^{\: y}}{k'})$}
and pick $\vec{g}_{n+1}^{\: y}$ such that for $s''':=\res{y}{k'+1}\in T$ we have $\vec{g}_{n+1}^{\: y}= x_{s'''}$ and $\res{y}{k'+1}\subsetneq x_{s'''}$.

\noindent Then, we are back to (a) through (d) with $\res{x}{m'+1}$, $\vec{g}_{n+1}^{\: x}$, $s''$, $\res{y}{k'+1}$, $\vec{g}_{n+1}^{\: y}, s'''$ replacing $\res{x}{m}$, $\vec{g}_{n}^{\: x}, s, \res{y}{k}, \vec{g}_{n}^{\: y}, s'$, respectively.

\noindent This finishes the construction of $x, \vec{g}^{\: x}, y,  \vec{g}^{\: y}$.
%%%%%%%%%%%%%%%%%%%%%%%%%
\noindent For every $n<\w$, $r(n)=1-\vec{g}_n^{\: x}(m')$, where $m'$ is maximal such that $\res{x}{m'}= \res{\vec{g}_{n}^{\: x}}{m'}$.  This shows (1*) on p.\,\pageref{1s}. 

\noindent To show (2*) on p.\,\pageref{1s}, notice that $\vec{g}_{n}^{\: y}= \vec{g}_{j}$ for the least $j$ such that $\res{y}{m'}= \res{\vec{g}_{j}}{m'}$, where $m'$ is maximal with $\res{x}{m'}= \res{\vec{g}_{n}^{\: x}}{m'}$; also, $\vec{g}_{n+1}^{\: x}= \vec{g}_{i}$ for the least $i$ such that $\res{x}{k'}= \res{\vec{g}_{i}}{k'}$, where $k'$ is maximal with $\res{y}{k'}= \res{\vec{g}_{n}^{\: y}}{k'}$.  

\noindent We have shown that $P\cap (\VV\smallsetminus W)\neq \varnothing$.  

\noindent Let us now prove the full theorem, varying the argument above.  By recursion on the length of $s\in {}^{<\w}2$ we construct $x^s, y^s\in T$ and subsequences $\vec{g}^{\: x^s}, \vec{g}^{\: y^s}$ of $\vec{g}$ such that

\begin{enumerate-(1)}\item $x^{s^\frown 0}, x^{s^\frown 1}$ and $y^{s^\frown 0}, y^{s^\frown 1}$ are incompatible;
\item $x^s\subsetneq x^{s'}$, $y^s\subsetneq y^{s'}$  for $s\subsetneq s'$;
\item $\vec{g}^{\: x^s}, \vec{g}^{\: y^s}$ are sequences of elements from $\vec{g}$, in fact from $\{x_s: s\in T\}$, of length $\lh(s)+1$;
\item $\vec{g}^{\: x^s}\subsetneq \vec{g}^{\: x^{s'}}$, $\vec{g}^{\: y^s}\subsetneq \vec{g}^{\: y^{s'}}$ for $s\subsetneq s'$;   \item if for $z\in {}^\w2$ we write $v^z=\bigcup\{v^s: s\subseteq z\}$, where $v\in\{x,y\}$, we have also\[\vec{g}^{\: x^z}= \bigcup\{\vec{g}^{\: x^s}: s\subseteq z\},\: \:  \vec{g}^{\: y^z}=\bigcup\{ \vec{g}^{\: s}: s\subseteq z\}\]\vspace{-1cm}
\item \label{5a}for all $z, z'\in{}^\w2$:\begin{enumerate-(5a)}\item $r\leq_T x^z, \vec{g}^{\: y^{z}}$, and
\item  $\vec{g}^{\: x^{z}}, \vec{g}^{\: y^{z'}}\leq_T x^z, y^{z'}, \vec{g}$.\end{enumerate-(5a)}
\end{enumerate-(1)}
In particular, $r\leq_T x^z, y^{z'}, \vec{g}$ for all $z, z'\in{}^\w2$.  But then $\{x^z: z\in{}^\w2\}\subseteq\VV\smallsetminus W$ or $\{y^z: z\in{}^\w2\}\subseteq \VV\smallsetminus W$, because if $x^z, y^{z'}\in W$ we would have $r\in W$.  By (1), both $\{x^z: z\in {}^\w2\}$ and $\{y^z: z\in{}^\w2\}$ are perfect, so one of them is a perfect set $\bar{P}\subseteq P$ consisting entirely of reals in $\VV\smallsetminus W$, as desired.

\noindent The construction of $x^s, \vec{g}^{\: x^s}, y^s, \vec{g}^{\: y^s}$ is basically as above, just building in (1).   Again, we start out with $x^{\varnothing}=\varnothing= y^\varnothing, \vec{g}^{\: x^{\varnothing}}=\langle\vec{g}_0\rangle= \vec{g}^{\: y^{\varnothing}}$.  Suppose we already have defined $x^s, \vec{g}^{\: x^s}, y^s, \vec{g}^{\: y^s}$ for all $s\in{}^{<\w}2$ of length $\leq n$.  

\noindent Fix $s$ of length $n$, and let us define $x^{s^\frown 0}$,\, $\vec{g}_{n+1}^{\:x^{s^\frown 0}}$,\, $x^{s^\frown 1}$,\, $\vec{g}_{n+1}^{\: x^{s^\frown 1}}$.  Let $j=\max\{\bar{j}: \vec{g}_{n}^{\: y^t}=\vec{g}_{\bar{j}},\: \lh(t)=n\}$, and pick $m'>\max\{\lh(x^t),\lh(y^t): \lh(t)=n\}$ such that $\res{\vec{g}_l}{m'}\neq \res{\vec{g}_{l'}}{m'}$ for all $l, l'\leq j$, $l\neq l'$ and $m_1>m_0\geq m'$ are both such that $\res{\vec{g}_n^{\: x^s}}{m_0}$, $\res{\vec{g}_n^{\: x^s}}{m_1}$ are splitting nodes in $T$ and $\vec{g}_n^{\: x^s}(m_0)\neq r(n)\neq \vec{g}_n^{\:x^s}(m_1)$.

\noindent Then set\begin{align*}x^{\:s^\frown0}&=\res{\vec{g}_n^{\:x^{s^\frown 0}}}{m_0}^\frown r(n)\\
x^{\:s^\frown1}&=\res{\vec{g}_n^{\:x^{s^\frown 1}}}{m_1}^\frown r(n)
\end{align*}
and pick $\vec{g}^{\: x^{s^\frown0}}_{n+1}, \vec{g}^{\: x^{s^\frown1}}_{n+1}$ such that there are $s'',\: \bar{s}''\in T$ with $x^{s^\frown 0}\subsetneq x_{s''}= \vec{g}^{x^{s^\frown0}}_{n+1}$, $x^{s^\frown 1}\subsetneq x_{\bar{s}''}=\vec{g}_{n+1}^{x^{s^\frown1}}$.

\noindent This defines all $x^t$, $\vec{g}_{n+1}^{x^t}$, $\lh(t)= n+1$.  Again, fix $s$ of length $n$, and let us define $y^{s^\frown0}$, $\vec{g}^{\: y^{s^\frown0}}_{n+1}$, $y^{s^\frown1}$, $\vec{g}^{\: y^{s^\frown1}}_{n+1}$.

Let $i=\max\{\bar{i}: \vec{g}^{x^{t}}_{n+1}= \vec{g}_{\bar{i}}, \lh(t)=n+1\}$ and pick $k'>\max\{\lh(y^{\bar{t}}), \lh(x^t):\lh(\bar{t})=n, \lh(t)=n+1\}$, such that $\res{\vec{g}_l}{k'}\neq \res{\vec{g}_{l'}}{k'}$ for $l,l'\leq i$, $l\neq l'$, and $k_1>k_0\geq k'$ are both such that $\res{\vec{g}_n^{\: y^s}}{m_0}$, $\res{\vec{g}_n^{\: y^s}}{m}$ are splitting nodes in $T$.

\noindent Then set
\begin{align*}y^{\:s^\frown0}&=\res{\vec{g}_n^{\:y^{s^\frown 0}}}{k_0}^\frown (1-\vec{g}_n^{\:s}(k(_0))\\
y^{\:s^\frown1}&=\res{\vec{g}_n^{\:y^{s^\frown 1}}}{k_1}^\frown (1-\vec{g}_n^{\: s}(k_1))
\end{align*}
and pick $\vec{g}_{n+1}^{\: y^{s^\frown0}}$, $\vec{g}_{n+1}^{\: y^{s^\frown1}}$ such that there are $s''', \bar{s}'''\in T$ with $y^{s^\frown0}\subsetneq x_{s'''}=\vec{g}^{\: y^{s^\frown0}}, y^{s^\frown1}\subsetneq x_{\bar{s}'''}=\vec{g}^{\: y^{s^\frown1}}$.

\noindent This defines all $y^t, \vec{g}^{\: y^t}_{n+1}$ where $\lh(t)=n+1$.\:  This finishes the construction.

\noindent The proofs of items (6-a) and (6-b) on p.\,\pageref{5a} are like the proofs of ($1^*$) and ($2^*$) on p.\,\pageref{1s}:\:\:  for each $n$, $r(n)=x^z(m)$, where $m$ is largest such that $\res{x^z}{m}=\res{\vec{g}_n^{\: x^z}}{m}$.  This shows (6-a).  Moreover, $\vec{g}_n^{\: y^z}=\vec{g}_j$ for the least $j$ such that $\res{y}{m'}=\res{\vec{g}_j}{m'}$ where $m'$ is maximal with $\res{x^{z'}}{m'}=\res{\vec{g}_n^{\: x^{z'}}}{m'}$.  Also, $\vec{g}_{n+1}^{\: x^{z}}= \vec{g}_i$ for the least $i$ such that $\res{x^z}{k'}=\res{\vec{g}_i}{k'}$ where $k'$ is maximal with $\res{y^{z'}}{k'}= \res{\vec{g}_n^{\: y^{z'}}}{k'}$.  This shows item (6-b).
\end{proof}

\subsection{Side-by-side product of Sacks forcing and its properties}

This section recapitulates well-known facts about Sacks forcing.

\begin{definition}\label{sacksdef}  Sacks forcing $\sa$ is defined in the following way.
\[\sa= \{T: T \text{\: is a perfect tree on 2}\}\]
For $S, T\in \sa$ we stipulate $S\leq T$ if and only if $S\subseteq T$.
If $S\in\sa$ and $p\in S$, we define the subtree $S_p=\{ t\in S: t\subset p \text{ or } p\subset t\}$ 

A node $p\in T$ is called a splitting node if $p^\smallfrown0, p^\smallfrown1\in T$.  The set of splitting points of $T$ is denoted by $\Spl(T)$.\:  We define  $\stem(T)$ as the unique element in $\Spl(T)$ comparable with any other node of $T$.  A node $p\in T$ is in $\Spl_n(T)$ if $p\in \Spl(T)$ and $p$ has exactly $n$ predecessors in $\Spl(T)$.  In particular, $\Spl_0(T)=\{\stem(T)\}$.  Notice that for $T\in \sa$, $\card{\Spl_{n}(T)}=2^{n}$. 

For every $n\in\w$ and $S \in \sa$ we write $\Lev{n}{S}= \{ t \in S : \exists s \in {\Spl_{n}(S)} \, t \subset s \}$, and for $S, T\in \sa$ we stipulate $S\leq_n T$ if and only if $S\leq T$ and $\Lev{n}{S}=\Lev{n}{T}$.
\end{definition}

\begin{definition}\index{forcing notions!side-by-side Sacks $\sa_\kappa$}If $\kappa$ is an ordinal and $X \subset \kappa$ (e.g., $X=\kappa$), let $\sa_X$ be the $\kappa$-side-by-side countable support product of Sacks forcing, i.e., $\sa_X$ is the set of all functions $p:X\to \sa$ such that $\supp{p}:=\{\alpha\in X: p(\alpha)\neq 1_\sa\}$ is at most countable.  If $p, q\in\sa_X$, we stipulate
\centm{$p\leq q \iff  \forall \alpha<\kappa ( p(\alpha)\leq_\sa q(\alpha))$}This implies in particular that $\supp{q}\subseteq \supp{p}$.\end{definition}
\noindent For now we are only interested in the case that $X = \kappa$ is a cardinal,
the more general case will only show up in the proof of Lemma \ref{m2}. If $g$ is $\sa_\kappa$-generic over $\VV$, and $\alpha<\kappa$, then
\centm{$s_\alpha=\bigcup_{p\in g}\stem{p(\alpha)}$}
is a real which is $\sa$-generic over $V$.  Therefore forcing with $\sa_\kappa$ adds $\kappa$-many Sacks reals which are independent over the ground model, i.e. for any $A\subset \kappa$ in $V$, 
\centm{${}^{\w}2^{V[\langle x_\alpha:\, \alpha\, \in\, A\rangle]}\cap {}^{\w}2^{V[\langle x_\alpha:\,\alpha\, \in\, \kappa\smallsetminus A\rangle]}= {}^\w2^V$}  

The product forcing $\sa_\kappa$ has properties very similar to those of $\sa$.  By defining a suitable notion of levels and fusion, it can be shown that $\sa_\kappa$ satisfies the Baumgartner Axiom A\footnote{For the details,  see \cite[\S 6]{gesc}} and therefore it  is proper and does not collapse $\w_1$.  For our purposes, the most remarkable property of $\sa_\kappa$ is that it inherits from $\sa$ also  the so called \emph{Sacks property}.  
 
\begin{definition}Let $g:\w\to\w$ be an increasing function.  We say $F:\w\to[\w]^{<\w}$ is a $g$-\emph{slalom} if $|F(n)|\leq g(n)$ for all $n\in\w$.\index{g@$g$-slalom}
\end{definition}

\begin{definition}Let $\PP$ be a forcing notion and suppose $g\in{}^\w\w\cap V$ is an increasing function.\:   We say that $\PP$ has the \emph{Sacks property}\index{forcing!notion!Sacks property} if whenever $G$ is $\PP$-generic over $\VV$, for every $f\in{}^\w\w\cap\VV[G]$ there exists a $g$-slalom $F\in \VV$, such that $V[G]\models \forall n (f(n)\in F(n))$.\footnote{For equivalent definitions of Sacks property, the reader can see \cite[Fact 6.35]{goldstern}.} \end{definition}

\begin{lemma}\label{7}Let $\kappa$ be a cardinal.  Suppose that $p\in\sa_{\kappa}$ and for $\theta \gg\kappa$ let $X\prec \VV_\theta$ be a countable elementary substructure with $p, \sa_\kappa\in X$.  Let $\seq{\tau_n}{n<\w} \in V$ be a sequence of terms for ordinals, $\{\tau_n: n<\w \}\subseteq X$ (possibly but not necessarily $\seq{\tau_n}{n<\w}\in X$).   Then, there is some $q\leq p$ and some $F:\w\to  [X\cap\OR]^{<\w}$, $F\in V$, such that for all $n<\w$:
\begin{enumerate-(1)}\item $q\Vdash \tau_n\in (F(n))^\vee$, 
\item $ \card{F(n)}\leq 2^{2n}$, and
\item $F(n) \subset X$.
\end{enumerate-(1)}
\end{lemma}

\begin{proof}Suppose that $\alpha=X\cap \w_1$.  Since $\supp{p}$ is an element of $X$, $\supp{p}$ also is a subset of $X$.  \:Let $e:\w\bij \alpha$ be a fixed bijection.  We aim to produce a sequence $\seq{p_n}{n<\w}$ such that $p_0=p$ and $p_{n+1}\leq p_n$, $p_n\in X$ for all $n\in\w$.   In this way, we also will have $\supp{p_n}\subseteq \alpha$ for every $n<\w$.   Suppose $p_n$ is already defined.  Working in $X$, we shall produce $p_{n+1}\leq p_n$ such that for all $k<n$, 
\begin{enumerate-(i)}
\item $p_{n+1}(e(k))\leq_n p_n(e(k))$, and
\item there is some $a_n\in[X\cap\OR]^{\leq2^{2n}}$ such that $p_{n+1}\Vdash \check{\tau}_n\in\check{a}_n$.
\end{enumerate-(i)}

\noindent The condition $q$ defined as $q(e(k))= \bigcap_{n<\w} p_n(e(k))$ for each $k<\w$ and the function $F$ given by $F(n)=a_n$ satisfy the conclusion of our lemma.

\noindent We may produce $p_{n+1}$ by means of some sequence $\seq{q_m}{m\leq 2^{2n}}$ defined as follows inside $X$.  Let $q_0=p_n$.\:  Fix some enumeration $\seq{\vec{s}_m}{m < 2^{2n}}$ of all tuples $\vec{s}=(s_{e(0)},\dots s_{e(n-1)})$ such that $s_{e(k)}\in \Lev{n}{p_n(e(k))}$ for all $k<n$.

Suppose $m<2^{2n}$ and $q_m$ has been chosen. We aim to define $q_{m+1}$. 
Write $\vec{s}_m = (s_{e(0)},\dots s_{e(n-1)})$. For each $k<n$, let $\bar{m}_k\leq m$ be maximal such that $s_{e(k)}\in q_{\bar{m}_k}$, and define $\bar{q}$ in such a way that $\supp{\bar{q}}=\supp{q_m}$  and
\[\bar{q}(\xi)=
\begin{cases} (q_{\bar{m}_k}(e(k)))_{s_{e(k)}} &\text{ if $\xi=e(k)$}\\
q_m(\xi)& \text{ if $\xi\neq e(k)$ for all $k<n$ }\end{cases}\]

\noindent Let $q_{m+1}\leq\bar{q}$ be a condition deciding $\check{\tau}_n$, and put the $\xi\in X\cap\OR$ with $q_{m+1}\Vdash \check{\tau}_n=\check{\xi}$ into $a_n$.  This defines $\seq{q_m}{m\leq2^{2n}}$.\:   Let us define $p_{n+1}$ as follows.  For each $k<n$ and $s\in\Lev{n}{p_{n}(e(k))}$, let $\bar{m}_{k,s}\leq m$ be maximal such that $s\in q_{\bar{m}_{k,s}}(e(k))$.  Then  $(q_{\bar{m}_{k,s}}(e(k)))_s=  q_{\bar{m}_{k,s}}(e(k))$.

Let $p_{n+1}$  have the same support as $q_{2^{2n}}$ and

\[p_{n+1}(\xi)=
\begin{cases} \bigcup \{ q_{\bar{m}_{k,s}}(e(k)) : s\in\Lev{n}{p_n(e(k))} \} &\text{ if $\xi=e(k)$}\\
q_{2^{2n}}(\xi)& \text{ if $\xi\neq e(k)$ for all $k<n$ }\end{cases}\]
It is easy to see that this sequence is as desired.\end{proof}

\begin{corollary}For every cardinal $\kappa$ the countable support product \label{2}$\sa_{\kappa}$ satisfies the Sacks property.  \end{corollary}

\begin{proof}  Let $f\in{}^\w\w\cap V^{\sa_\kappa}$ and let $p\in \sa_{\kappa}$ such that $p\Vdash\tau\in{}^\w\w$ where $\tau$ is a $\sa_\kappa$-name for $f$.  Let $\theta> 2^{2^{\kappa}}$ and let $X\prec\V_\theta$ be a countable elementary substructure such that $p,\tau, \sa_\kappa\in X$.  Suppose that $\alpha=X\cap\w_1$.\:    By Lemma \ref{7},  there is a $2^{2n}$-slalom $F:\w\to[\w]^{<\w}$ in $V$ and a condition $q\leq p$ with $\supp{q}\subseteq \alpha$ such that \centm{$q\Vdash \forall n \, \tau(n)\in {{F(n)}}^\vee.$}% \text{ and } \card{f(n)}\leq 2^{2n}]$$
Given any increasing function $g : \omega \rightarrow \omega$, a simple variant of the argument for Lemma \ref{7} with an appropriate bookkeeping produces a $g$-slalom $F$ and
a condition $q \leq p$ with the same properties.  
Therefore $\sa_\kappa$ has the Sacks property.\end{proof}

\begin{corollary}\label{its-proper} For every cardinal $\kappa$, the countable support product $\sa_\kappa$ is a proper forcing. If
$g$ is $\sa_\kappa$-generic over $V$ and if $x \in {}^\omega 2 \cap V[g]$,
then there is some $\tau \in V^{{\sa_\kappa}}$ which is %hereditarily 
countable
in $V$ such that $x = \tau^g$. \end{corollary}
\begin{proof} First part:
Let $p\in \sa_\kappa$.  Suppose that $\theta\gg\sa_\kappa$ and let $N\prec H_\theta$ be a contable substructure with $\sa_\kappa\in N$, $p\in N$.  

%\noindent %Suppose that $N\models \text{``$\dot{\alpha}$ is a name for an ordinal''}$.  
 Let $\{\tau_n: n\in\w\} \in V$ be an enumeration of all $\sa_\kappa$-names for ordinals in $N$.  By lemma \ref{7}, there exists some $q\leq p$ and some $F:\w\to [N\cap\OR]^{<\w}$ in $V$ such that for all $n\in\w$, \centm{$q\Vdash \tau_n\in F({n})^\vee \subset \check{N}.$}
%Since $\dot{\alpha}=\alpha_m$ for some $m$ and $F(m)\in[N\cap\OR]^{2^{2m}}$, it 
I.e., $q\Vdash \dot{\alpha}\in {\check N}\cap\OR$ for every $\sa_\kappa$-name 
${\dot \alpha} \in N$ for an ordinal.   
This implies that $\sa_\kappa$ is proper. 

Second part: Let $x = \sigma^g$, where $\sigma = \bigcup \{ \{ (n,h)^\vee \} \times A_{n,h} \colon (n,h) \in \omega \times 2 \} \in V^{\sa_\kappa}$ and for each $(n,h) \in \omega \times 2$, $A_n$ is a maximal
antichain of $p \in \sa_\kappa$ such that $p \Vdash \sigma({\check n}) = {\check h}$.
In $V[g]$, for each $n<\omega$ there is some unique $h = h_n \in 2$ and $p = p_n
\in\sa_\kappa$ such that $p \in A_{n,h} \cap g$. Let $X \supset \{ p_n \colon n<\omega \}$, where $X \in V$ is countable in $V$. Then $\tau = \bigcup \{ \{ (n,h)^\vee \} \times (A_{n,h} \cap X) \colon (n,h) \in \omega \times 2 \}$ is as desired. 
\end{proof}

\cite{kanovei-sacks} gives more information on how reals in $V^{\sa_\kappa}$ may be
represented.

\section{Lusin and Sierpi\'{n}ski sets in the Sacks model}

%\begin{definition}Sacks forcing $\sa$ is the set of all perfect trees $T\subset 2^{\w}$ ordered by stipulating $T\leq S$ if and only if $S\subset T$.\end{definition}

Let $\sa_{\w_1}$ be the countable support product of $\w_1$-many copies of Sacks forcing.  From the fact that $\sa_{\w_1}$ has the Sacks property we shall show that in the generic extension obtained after forcing with $\sa_{\w_1}$ the Lusin and Sierpi\'{n}ski sets in the ground model are also Lusin and Sierpi\'{n}ski sets in the generic extension.  

We use the following result. 

\begin{lemma}\label{claimfact3}Let $N\subseteq {}^\w2$ be null and  let $\{\varepsilon_n:n\in\w\}$ be a sequence of positive reals.  Then there is a sequence $\langle C_n\subseteq {}^{\w}2:n\in\w\rangle$ of finite unions of basic open sets such that
\begin{enumerate-(i)}\item for all $n<\w$, $\mu(C_n)<\varepsilon_n$ and
\item $N\subseteq\bigcup_{n\in\w}C_n$
 \end{enumerate-(i)}\end{lemma}
 
 \begin{proof}Since $N$ is null, there is a collection of basic open sets $\{O_n:n\in\w\}$ such that $N\subset\bigcup\{ O_n:n\in\w\}$ and $\mu(\bigcup_{n\in\w}O_n)<\varepsilon_0$. 
 
 Then let $k(n)=\min\{m: \mu(\bigcup_{i\geq m}O_i)<\varepsilon_n\}$.  Without loss of generality, we can assume that the sequence $\langle \varepsilon_n:n\in\w\rangle$ is decreasing, so $k$ is monotone.  We have $k(0)=0$.  Then for each $n$ set \[C_n=\bigcup\{O_i: k(n)\leq i<k(n+1)\}.\] It is straightforward to see that the collection $\{C_n: n\in\w\}$ satisfies (i) and (ii).\end{proof}
 
%In the Sacks model, $\sa_{\w_1}$,  new null and meager sets are covered by older ones.  This property, is due to the Sacks property, as stated in the next result.

\begin{lemma}\label{3}Let $\PP$ be a forcing notion satisfying the Sacks property and let $G$ be a $\PP$-generic filter over $V$.  Then: %Let $p\in\PP$.  Then, the following holds: 

\begin{enumerate-(1)}
\item For every null set $N\subseteq {}^\w\w$ in $V[G]$ there is a $G_\delta$-null set $\bar{N}\subseteq {}^\w\w$ coded in $V$ such that $N\subseteq \bar{N}$.
%\item  If $\dot{N}$ is a $\PP$-name for a null set in $V^\PP$,   or, equivalently, a $\PP$-name for a Borel code of a null set, there exist $q\leq p$ and $\bar{N}$ a (Borel code for a ) $G_\delta$-null set in $V$ such that $q\Vdash \dot{N}\subset \check{\bar{N}}$;\footnote{Here $\bar{N}$ stands for the interpretation of the Borel code (of the ground model) in the generic extension.}
\item Similarly, for every meager set $M\subseteq {}^\w\w$ in $V[G]$, there is a meager set  $\bar{M}\subseteq{}^\w\w$ coded in $V$ such that $M\subseteq \bar{M}$. %, or equivalently, a $\PP$-name for a Borel code of a meager set,  there exist $q\leq p$ and a (Borel code for a) $F_\sigma$-meager set  $\bar{M}$ in $V$ such that $q\Vdash \dot{M}\subset \check{\bar{M}}$. 
\end{enumerate-(1)}

\end{lemma}

\begin{proof}We prove the statement (1).  Let $\varepsilon>0$.   First, let us fix in $V$ an enumeration $\{C_n: n<\w\}$ of all finite unions of basic open sets in ${}^\w2$. 

\noindent Let $N\subseteq {}^\w\w$ be a null set in $V[G]$.  By \ref{claimfact3} there is a function $f:\w\to\w$ in $V[G]$ such that \[N\subseteq\bigcup_{n\in\w}C_{f(n)}\text{\:\:\:\:\: and \:\:\:\:\:}\mu(C_{f(n)})\leq \frac{\varepsilon}{2^{2n+1}},\: n\in\w\]  
Since $\PP$ has the Sacks property, there is a $2^n$-slalom $F:\w\to[\w]^{<\w}$ such that for every $n\in\w$, $f(n)\in F(n)$.  Set
\[\bar{N}=\bigcup_{n\in\w}\bigcup\{C_k:k\in F(n) \text{   and  } \mu(F_k)\leq \frac{\varepsilon}{2^{2n+1}}\}\]

\noindent Since only ground model parameters are used in the definition of $\bar{N}$, $\bar{N}$ is a open set coded in the ground model.   Note also that $N\subseteq\bar{N}$.  Now since $|F(n)|\leq2^n$ for each $n\in\w$ it  follows that \[\mu(\bigcup\{C_k:k\in F(n) \text{ and }\mu(C_k)\leq \frac{\varepsilon}{2^{2n+1}} \})\leq 2^n\cdot \frac{\varepsilon}{2^{2n+1}}=\frac{\varepsilon}{2^{n+1}}\]
Therefore $\mu(\bar{N})\leq\sum_{n\in\w}\frac{\varepsilon}{2^{n+1}}=\varepsilon$.  Since $\varepsilon$ was taken arbitrarily, it follows that $\bar{N}$ is a null subset coded in $V$. \qedhere \end{proof}

\begin{remark}Let $\mathcal{N}$ and $\mathcal{M}$ stand for the null and meager ideals over ${}^\w\w$ respectively.  Since $\text{add}(\mathcal{N})\leq \text{add}(\mathcal{M})$ and $\text{cof}(\mathcal{M})\leq \text{cof}(\mathcal{N})$, if a forcing notion $\PP$ satisfies item (1) above, then $\PP$ satisfies (2) as well.  \end{remark}

\begin{corollary}\label{4}If $\PP$ has the Sacks property, then $\PP$ preserves Luzin and Sierpi\'{n}ski sets.\end{corollary}

\begin{proof}Suppose that there is a Luzin set $\Lambda$ in $V$ and let $G$ be $\PP$-generic over $V$. First, observe that, since $\w_1$ is not collapsed by $\PP$, $\Lambda$ remains to be non countable in $V[G]$.  Now,  let $M$ be a (Borel code for a) meager set in $\VV[G]$.  In view of Lemma \ref{3}, there is a (Borel code) for a $G_\delta$-null set $\bar{M}$ in $V$ such that $V[G]\models M\subset \bar{M}$.  Thus, since $V\models |\Lambda\cap \bar{M}|\leq \w$, it follows that $V[G]\models |\Lambda\cap M|\leq \w$.  Hence, \centm{$V[G]\models\Lambda\text{ is a Luzin set.}$}The proof of the preservation of  Sierpi\'{n}ski sets is completely analogous.  \end{proof}

\section{Adding generically a Burstin basis}

%Suppose that $g$ is $\sa_{\w_1}$-generic over $L$ and consider the model $L[g]$.  
%Let us define a partial order $\PP_{B}$ adding a Burstin basis on $L[g]$.  
We now define a partial order $\PP_{B}$ generically adding a Burstin basis.
 
 \begin{definition}\index{forcing notions!adding a Burstin basis $\PP_B$}\label{defn_BF}
 We say $p\in \PP_{B}$ if and only if there exists $x\in\RR$ such that 
\begin{enumerate}\item $p\in \LL[x]$, and
\item $\LL[x]\models \text{``$p$ is a Burstin basis."}$
\end{enumerate} 
\noindent We stipulate $p\leq_{\PP_B}q$ iff $p\supseteq q$.  
\end{definition}
Notice that by Theorem \ref{burstin} we have $\PP_{B}\neq\varnothing.$

If ${\mathbb R} \cap V \subset L[x]$ for some real $x$, then $\PP_B$ has a dense set of atoms.
We are interested in situations where the set of all reals is not constructible from a single 
real.
Variants of $\PP_{B}$ will be discussed at the end of this chapter.

The following is an immediate consequence of Theorem \ref{gro}.

\begin{lemma}\label{perf} 
%Let $P\in\LL[x,y]$ be perfect and l
Let $x$, $y$ be reals such that $y \notin L[x]$, and let $\{z_0, z_1,\dots\}\in \LL[x,y]\cap[\RR]^{\w}$. Then \centm{$\text{\emph{span}}{(\RR\cap\LL[x]\cup\{z_0,z_1,\dots\})} \in (s^0)^{\LL[x,y]}$,} 
i.e., for every perfect set $P$ in $\LL[x,y]$ 
there is a perfect set $\bar{P}\subset P$, $\bar{P}\in\LL[x,y]$ such that \centm{$\bar{P}\cap\text{\emph{span}}{(\RR\cap\LL[x]\cup\{z_0,z_1,\dots\})}=\varnothing$}  \end{lemma}

\begin{proof}We may assume that if $z\in\Span{\RR\cap\LL[x]\cup\{z_0,z_1,\dots\}}$, then $z\in(\RR\cap\LL[x])+z_n$, for some $n<\w$.  Given $P\in L[x,y]$ a perfect set, we shall construct recursively a sequence $T_0\supseteq T_1\supseteq \cdots T_{n}\supseteq T_{n+1}\supseteq \cdots$ of perfect trees, such that  \begin{enumerate-(1)}\item $P=[T_0]$,
\item$\Lev{n}{T_{n+1}}=\Lev{n}{T_{n}}$ and,
\item $[T_{n+1}]\cap((\RR\cap\LL[x])+z_n)=\varnothing$.
\end{enumerate-(1)}

Let $T_0$ be the perfect tree such that $P=[T_0]$.    By Theorem \ref{gro} we have that $L[x,y]\models ``{}^\w2\cap L[x]\in s^0$''.    Since $P-z_0=\{x-z_0: x\in P\}$ is also perfect in $L[x,y]$, there is some $\tilde{P}\subset P-z_0$ perfect,  $\tilde{P}\in\LL[x,y]$,  such that $\tilde{P}\subseteq L[x,y]\smallsetminus\LL[x]$.   Therefore $P':=\tilde{P}+z_0\subseteq P$ is perfect and if $u\in \tilde{P}$ (equivalently, $u+z_0\in\tilde{P}+z_0=P'$), then $u\notin L[x]$, so $u+z_0\notin (\RR\cap L[x])+z_0$.  Thus, $P'\cap(\RR\cap L[x]+z_0)=\varnothing$.   Take then $T_1$ as the perfect tree such that $P'=[T_1]$.

Now suppose that we have constructed $T_0, T_1,\dots, T_n$ satisfying (1)-(3) above.   For any $s \in \Lev{n}{T_n}$ let us consider the subtree $(T_n)_{s}$ of $T_n$.  By the argument from the previous paragraph, there is some perfect set $P_{n,\,s}\subset [(T_n)_{s}]$ such that  $P_{n,\, s}\cap (\RR\cap L[x]+z_n)=\varnothing$.    Let \[P_{n+1}= \bigcup \{P_{n,\, s}: s\in \Lev{n}{T_n} \}.\]
Notice that $P_{n+1}\cap (\RR\cap L[x]+z_n)=\varnothing$, hence by taking $T_{n+1}$ as the perfect tree such that $P_{n+1}=[T_{n+1}]$ condition (3) holds.  Also, by construction, $\Lev{n}{T_{n+1}}=\Lev{n}{T_n}$.  

Now, set $T= \bigcap\{T_n: n\in\w\}$.  By condition (2), we have that $T$ is a perfect tree.\:  Thus $\bar{P}:=[T]$ is a perfect set such that $\bar{P}\cap\Span{\RR\cap\LL[x]\cup\{z_0,z_1,\dots\}}=\varnothing$, as required.\end{proof}

\begin{lemma}\label{5}Let $b\in\LL[x]$ be linearly independent, $x\in\RR$.  Let $y\in \RR\setminus \LL[x]$.  There is then some $p\supset b$, $p\in\LL[x,y]$ such that $\LL[x,y]\models ``\text{$p$ is a Burstin basis"}$.\end{lemma}

\begin{proof} Let $\seq{P_i}{i<\w_1}$ be an enumeration of all perfect sets of $\LL[x,y]$.  Working in $\LL[x,y]$ we define recursively $\seq{b_i}{i<\w_1}$ as follows.   Let $\{y_i: i<\w_1\}\in\LL[x,y]$ enumerate the reals of $\LL[x,y]$. Given $\{b_j:j<i\}$, we will have that $\bar{b}=\bigcup \{b_j:j<i\}$ is at most countable. By Lemma \ref{perf} there is some $\bar{P}\subset P_i$  perfect such that $\bar{P}\cap\Span{(\RR\cap\LL[x])\cup \bar{b}}=\varnothing$.  Pick $\bar{x}\in\bar{P}$ and set
\[b_i=\begin{cases}\bar{b}\cup\{\bar{x}\}&\text{ if $y_i\in\Span{(\RR\cap L[x])\cup\bar{b} \cup\{\bar{x} \} } $}\\
\bar{b}\cup\{\bar{x}, y_i\}&\text{ otherwise}\end{cases}\]
Finally, if  $c\in\LL[x]$ is such that $c\supseteq b$ and $\LL[x]\models \text{``$c$ is a Hamel basis"}$, take \centm{$p:=c\cup \bigcup\{b_i:i<\w_1\}$}
By construction $p$ is a Hamel basis for $L[x,y]$.  Moreover for each $i<\w_1$, $b_i\subset p$ hence $P_i\cap p\neq \varnothing$.  This shows that $p$ is a Burstin basis in $L[x,y]$.
\end{proof}

Lemma \ref{5} has the following immediate corollary, extendability for $\PP_B$:

\begin{lemma}\label{extendability}
If $p\in\PP_B$, say $\LL[x]\models \text{``$p$ is a Burstin basis,"}$ and if $y$ is a real not in $\LL[x]$, then there is some $q\leq_{\PP_B} p$ such that $q$ is a Burstin basis in $\RR\cap\LL[x,y]$.
\end{lemma}

Also, lemma \ref{5} shows that $\PP_B$ is countably closed under favourable circumstances. What is more than enough for our purposes is the following.
Hypothesis (1) of Lemma \ref{pb3} is satisfied e.g.\ if $V$ is a forcing extension
of $L$ via some proper forcing. Hypotheses (1) and (2) are certainly satisfied 
in $V=L[g]$, where $g$ is $\sa_{\w_1}$-generic over $\LL$, cf.\ Corollary \ref{its-proper}.

\begin{lemma}\label{pb3} Assume that 
\begin{enumerate}
\item[(1)] for every countable set $X$ of ordinals 
there is a set $Y \supset X$, $Y \in \LL$, such that $Y$ is countable in $\LL$, and
\item[(2)] there is no real $x$ such that $\RR \subset \LL[x]$.
\end{enumerate}
Then $\PP_B$ is $\w$-closed. In particular, forcing with $\PP_B$ does not add any new reals.
\end{lemma}

\begin{proof} Consider a sequence $( p_n: n<\w)$ of conditions in $\PP_B$ such that $p_{n+1}\leq_{\PP_B} p_n$ for all $n<\w$. For each $n<\omega$, let $x_n \in
\RR$ be such that $p_n\in\LL[x_n]$ is a Burstin basis for $\RR\cap\LL[x_n]$.
Pick $z \in \RR$ such that $x_n \in L[z]$ for all $n<\omega$.  

\medskip
{\bf Claim.}
There is some $x \in \RR$ such that $\{ p_n: n<\w \} \in \LL[x]$.

\medskip
To prove the claim, notice that $\{ p_n : n<\w \} \subset L[z]$. 
Let $F : {\rm OR} \rightarrow L[z]$ be bijective and definable over $L[z]$,
and let 
$X = \{ \xi : \exists n<\w \, F(\xi) = p_n \}$. By hypothesis (1) there is some
$Y \supset X$, $Y \in \LL$, and $Y$ is countable in $\LL$. Let $f \colon \omega
\rightarrow Y$ be bijective, $f \in \LL$, and write $x^* = f^{-1}\mbox{''}X$.
Then $x^* \subset \omega$ and $X = f \mbox{''} x^* \in \LL[x^*]$. But then
$\{ p_n: n<\w \} \in\LL[z,x^*]$, and if $x \in \RR$ is such that $\LL[z,x^*] \subset \LL[x]$,
then $x$ verifies the Claim.

Now let $b = \bigcup \{ p_n : n<\w \}$, let $x$ be as in the Claim, 
and let us make use of
hypothesis (2) to pick some $y \in \RR \setminus \LL[x]$. We have that $b \in L[x]$, so that by Lemma 
\ref{5} we can extend the linearly independent set $b$ to a Burstin basis $p$ over $\LL[y]$.  Then, for every $n<\w$ we have that $p\leq_{\PP_B}p_n$, so $\PP_B$ is $\w$-closed.
\end{proof}

\begin{notation}For $\vec{x}$, $\vec{y}$  two real vectors of the same lenght,  let $\vec{x}\cdot\vec{y}:=\sum_{i<\lh(x)}x_iy_i$.
\end{notation}

\begin{remark}\label{pbf}We have that {\begin{align*}\label{eq1}p\in\PP_B &\iff\exists x(L[x]\models \text{``$p$ is a Burstin basis''})\\
&\iff\exists \vec{x}\in [p]^{<\w}\exists \vec{q}\in [\QQ]^{<\w}(\forall y\in\RR^{L[\vec{q}\cdot\vec{x}]}\exists \vec{p}_y\in[p]^{<\w}\exists \vec{q}_y\in[\QQ]^{<\w}\\
&\hspace{10mm} y=\vec{q}_y\cdot\vec{p}_y \wedge \forall \vec{z}\in[p]^{<\w}\forall\vec{q}\in[\QQ]^{<\w}(\vec{q} \cdot \vec{z}=0 \rightarrow \vec{q}=\vec{0})\wedge\\
&\hspace{10mm}  L[\vec{q}\cdot\vec{x}] \models \mbox{``} P \cap p\neq\varnothing \mbox{ for every perfect set } P \mbox{''} ) \end{align*}}
%\noindent where $\{T_\alpha:\alpha<\w_1\}$ is an enumeration of all perfect sets in %$L[\vec{q}\cdot\vec{x}]$ and $e:\RR\longleftrightarrow\w_1$ is a (recursive) bijection %in $ L[\vec{q}\cdot\vec{x}]$\footnote{Recall that for every $x\subseteq\w_1$, %$L[x]\models \ZFC+\CH$.}.
Since the matrix of this formula is $\Pi^1_{2}$ 
%formula of $\vec{x}, \vec{q},\vec{p}$, 
we have that  
\begin{equation}p\in \PP_B\iff \exists \vec{x}\in [p]^{<\w}\exists \vec{q}\in  [\QQ]^{<\w}\psi(\vec{x},\vec{q},p)\label{eq1}\end{equation}where $\psi$ is  $\Pi^1_2$.  
%Thus \centm{$p\in\PP_B\iff \theta(p)$} where $\theta$ is a $\Sigma^1_3$-formula %with no parameters.
\end{remark}

\begin{remark}\label{gb}
In what follows, we will call
%
%Let $h$ be a $\PP_B$-generic filter over $L$. According with the definition of the %forcing $b:=\bigcup h$ is the Burstin basis for $\RR^{L}$ added by $h$. Since %$\dot{h}=\{(\check{q}, q):q\in \PP_B\}$ is the canonical $\PP_B$-name for $h$, it is %straightforward to see that  
\centm{$\dot{b}:=
\{ ({\check x},p) : x \in p \in \PP_B \}$}
%
%\{(\sigma, q): \exists r\in\PP_B [(\sigma, \one)\in\check{r}\wedge q\leq_{\PP_B} %r]\}$} is a $\PP_B$-name for the generic Burstin basis $b$.  In this setting  we say that %$\dot{b}$ is 
the \emph{canonical name} for the generic Burstin basis $b$. 
%\:\: Note that $\dot{b}=\{(\sigma,q): \theta(q)\wedge \exists r[\theta(r)\wedge %(\sigma,\one)\in\check{r}\wedge r\subseteq q]\}$ where $\theta$ is as in \ref{pbf}.  %Therefore \centm{$\dot{b}=\{(\sigma,q): \theta^*(\sigma,q)\}$} where $\theta^*$ is %a $\Sigma^1_3$-formula with no parameters and hence 
By the previous remark, 
{\begin{align*}
({\check x},p)\in \dot{b} &\iff 
x \in p \wedge \exists \vec{x}\in [p]^{<\w}\exists \vec{q}\in  [\QQ]^{<\w}\psi(\vec{x},\vec{q},p)\\
&\iff \theta(x,p){\rm , }
\end{align*}}
where $\theta$ is $\Sigma^1_3$, and
``$({\check x},p)\in \dot{b}$'' is absolute between transitive class sized models of set theory.\end{remark}

Let us discuss some variants of ${\mathbb P}_B$.
\begin{definition}\label{defn_hamel_basis_forcing_1}
We say $p\in \PP_{H}$ if and only if there exists $x\in\RR$ such that 
\begin{enumerate}\item $p\in \LL[x]$, and
\item $\LL[x]\models \text{``$p$ is a Hamel basis."}$
\end{enumerate} 
\noindent We stipulate $p\leq_{\PP_B}q$ iff $p\supseteq q$. 
\end{definition}

If ${\mathbb R} \cap V \subset L[x]$ for some real $x$, then like $\PP_B$, $\PP_H$ has a dense set of atoms.
If there is no real $x$ with ${\mathbb R} \cap V \subset L[x]$, then the content 
of Lemma \ref{5} is exactly that $\PP_B$ is dense in $\PP_H$, which implies that
$\PP_H$ and $\PP_B$ will be forcing equivalent and forcing with
$\PP_H$ will not just add a Hamel basis but in fact a Burstin basis.

Hence if we aim to generically add a Hamel basis which in the extension
contains a perfect set, then forcing with $\PP_H$ won't work. E.g., let
$P \in L$ be a perfect set in $L$ which is also linearily independent, see 
\cite[Example 1, p.\ 477f.]{jones}. If $M \supset L$ is any inner model, then let us
write $P_M$ for $M$'s version of $P$. Then $P_M$ is perfect in $M$,
$P_M \cap L = P$, and by $\Pi^1_1$ absoluteness, $P_M$ is linearily independent
in $M$. We may then let $p\in \PP_{H}^P$ if and only if there exists $x\in\RR$ such that 
$p\in \LL[x]$, $p \supset P_{\LL[x]}$, and
$\LL[x]\models \text{``$p$ is a Hamel basis"}$;
$p\leq_{\PP_H^P}q$ iff $p\supseteq q$. If $p \in \PP_H^P \cap \LL[x] \subset \LL[y]$,
$x$, $y \in {\mathbb R}$, then $p \cup P_{\LL[y]}$ is linearily independent by 
$\Pi^1_1$ absoluteness, so that $\PP_H^P$ will generically add a Hamel basis which
contains the version of $P$ of the model over which we force. The proof of Lemma 
\ref{m2} will go through for $\PP_H^P$ instead of $\PP_B$.

The following forcing, ${\mathbb Q}_{H}$, is the obvious candidate for adding a Hamel basis.

\begin{definition}\label{defn_hamel_basis_forcing_2}
We say $p\in {\mathbb Q}_{H}$ if and only if $p$ is a countable
linearily independent set of reals. 
\noindent We stipulate $p\leq_{\PP_H}q$ iff $p\supseteq q$. 
\end{definition}

It is clear that in $\omega_1$ is inaccessible to the reals (i.e., ${\mathbb R} \cap L[x]$
is countable for all reals $x$), then ${\mathbb Q}_{H}$ is dense in $\PP_H$ (and hence
also in $\PP_B$), so that under this hypothesis all the three forcings are forcing equivalent
with each other.
On the other hand, in the absence of large cardinals, in contrast to $\PP_B$ and $\PP_H$
(see Lemma \ref{m2} below) forcing with ${\mathbb Q}_H$ over $L({\mathbb R})$
will add a well-ordering of ${\mathbb R}$, see Corollary \ref{key_corollary} below, 
so that ${\mathbb Q}_H$ would be
the wrong candidate for forcing a Hamel basis for our purposes. 
(The forcing ${\mathbb Q}_H$ would be called $P_{\psi}$ in \cite{larson_zapletal}, where
$\psi$ expresses linear independence, see \cite[Introduction]{larson_zapletal}.)

\begin{lemma}\label{doesntwork}
Let ${\vec x} = (x_\alpha : \alpha < \omega_1)$ be a sequence of 
pairwise distinct reals such that $\{ x_\alpha : \alpha < \omega_1 \}$ is linearily 
independent. Let $g$ be ${\mathbb Q}_H$-generic over $V$, and let $h=
\bigcup g$. Then inside $L({\mathbb R},{\vec x},h)$, there is a well-order of ${\mathbb R}$
of order type $\omega_1$. In particular, $L({\mathbb R},{\vec x},h)$ is a model of $\ZFC$.
\end{lemma}

\begin{proof}
Of course ${\mathbb Q}_H$ is $\omega$-closed, so that $V$ and $V[g]$ have
the same reals. Hence $h$ is a Hamel basis inside $L({\mathbb R},h)$.

Let $p \in {\mathbb Q}_H$, and let $x \subset \omega$.
There is a countable limit ordinal $\lambda$ such that $p \cup \{ x_{\lambda +n}
: n<\omega \}$ is linearily independent. Let $$q = p \cup
\{ x_{\lambda +n} : n \in x \} \cup \{ 2 \cdot x_{\lambda +n} : n \in \omega
\setminus x \}.$$ 
Then $q \in {\mathbb Q}_H$, $q \leq_{{\mathbb Q}_H} p$, and $x = \{ n<\omega
: x_{\lambda +n} \in q \}$.

In $L({\mathbb R},{\vec x},h)$ let us define $f : \pow(\omega) \rightarrow \omega_1$
by $f(x) =$ the least countable limit ordinal such that $x = \{ n<\omega
: x_{\lambda +n} \in h \}$. Trivially, $f$ is injective, and by the density argument
from the previous section $f$ is a well-defined total function. This shows that
in $L({\mathbb R},{\vec x},h)$, there is a well-order of ${\mathbb R}$
of order type $\omega_1$.

As there is a surjection $F : {\mathbb R} \times {\rm OR} \rightarrow L({\mathbb R},{\vec x},h)$
which is $\Sigma_1$-definable over $L({\mathbb R},{\vec x},h)$ from the parameters ${\mathbb R}$, ${\vec x}$, and $h$, the existence of a well-order of ${\mathbb R}$ inside $L({\mathbb R},{\vec x},h)$
yields that $L({\mathbb R},{\vec x},h)$ is a model of $\ZFC$.
\end{proof}

\begin{corollary}\label{key_corollary}
Assume that $\omega_1$ is not inaccessible the reals, let $g$ be ${\mathbb Q}_H$-generic over $V$, and let $h=
\bigcup g$. Then in $L({\mathbb R},h)$, there is a well-order of ${\mathbb R}$
of order type $\omega_1$ and $L({\mathbb R},h)$ is a model of $\ZFC$.
\end{corollary}

\begin{proof} By our hypothesis, there is a real $x$ such that we may pick 
${\vec x} \in L[x]$ and ${\vec x}$ is as in the hypothesis of Lemma \ref{doesntwork}.
\end{proof} 

\section{The main theorem}

The following Lemma is dual to Corollary \ref{key_corollary}.

 \begin{lemma}\label{m2}Let $g$ be $\sa_{\w_1}$-generic over $\LL$, let $h$ be $\PP_B$-generic $h$ over $\LL[g]$ and let $b = \bigcup h$ be the Burstin basis added by $h$.  Let 
\[W= \LL(\RR, b)^{\LL[g,h]}\]
Then $W\models \text{``\,There is no well-ordering of $\RR$\:''}$.\end{lemma}

\begin{proof} That $b$ is indeed a Burstin basis in $\LL[g,h]$ as well as in $W$ follows
from Lemmas \ref{extendability} and \ref{pb3}.

Let us assume for contradiction that 
\[\LL[g, h]\models \text{``$\varphi(\cdot\;, \cdot\;, \vec{x}, \vec{\alpha}, b)$ defines a well-ordering of ${}^\w2$"}\]
where $\vec{x}\in\RR\cap\LL[g,h]=\RR\cap\LL[g]$ and $\vec{\alpha}\in\OR$. 

Then, there is some $p\in h\subset\PP_B$ such that
\[p\dststile{\LL[g]}{\PP_B}\text{``$\varphi(\cdot\;, \cdot\;, \check{\vec{x}}, \check{\vec{\alpha}}, \dot{b})$ defines a well-ordering of ${}^\w2$"} \]where $\dot{b}$ is the canonical $\PP_B$-name for the generic Burstin basis $b$ as defined in  Remark \ref{gb}; but then we may rewrite this as
%.   Replacing $\dot{b}$ by the set defined using the $\Sigma^1_3$-formula $\theta$ as %in Remark \ref{gb} we can write  %As $\bar{p}\in \PP_B$ if and only if \eqref{eq1} on p.\,\pageref{eq1} holds  we can write $\dot{b}$ as being replaced by a $\Sigma^1_3$ formula with no parameters in a way that $(\bar{p}, \check{\bar{p}})\in\dot{b}$ is absolute between transitive class sized models of set theory. 
\[p\dststile{\LL[g]}{\PP_B}\text{``$\varphi(\cdot\;, \cdot\;, \check{\vec{x}}, \check{\vec{\alpha}}, \{({\check y},q): \theta(y,q)\})$ defines a well-ordering of ${}^\w2$,''} \]
with $\theta$ being the $\Sigma^1_3$ formula from Remark \ref{gb}.
We may pick $\xi<\w_1$ with $p, \vec{x}\in\LL[\res{g}{\xi}]$, see Corollary \ref{its-proper}.  Now since $\sa_{\xi}\times \sa_{\w_1\setminus \xi}$ is isomorphic to $\sa_{\w_1}$ via the the isomorphism $(p_0,p_1)\mapsto p_0\cup p_1$,  standard arguments show that $\res{g}{[\xi,\w_1)}$ is $\sa_{\w_1\setminus \xi}$-generic over $L[\res{g}{\xi}]$ and so we can write
\begin{equation}\label{eq2}p\dststile{\LL[\res{g\,}{\,\xi}][\res{g\,}{\,[\xi,\,\w_1)}]}{\PP_B} \text{``$\varphi(\cdot\;, \cdot\;, \check{\vec{x}}, \check{\vec{\alpha}}, \{({\check y},q): \theta(y,q)\})$ defines a well-ordering of ${}^\w2$''} \end{equation}

The following only uses that $\sa_{\w_1}$ is a countable support product of uncountably many copies of the same forcing.

\begin{claim}\label{homo}$\sa_{\w_1}$ is weakly homogeneous, i.e., given $p,p'\in\sa_{\w_1}$ there is an isomorphism $\pi:\sa_{\w_1}\to\sa_{\w_1}$ such that $p||\pi(p')$.\end{claim}
\begin{proof}
Let $p,p'\in\sa_{\w_1}$.   Since $\supp{p}$ is countable there is some $\gamma<\w_1$ such that $\supp{p}\subset \gamma$.  Set $\pi:\sa_{\w_1}\to\sa_{\w_1}$ defined as follows:\[\pi(r)(\beta)=\begin{cases} 1_{\sa} &\text{ if $\beta<\gamma$}\\
r(\alpha)&\text{ if $\beta=\gamma+\alpha$}\end{cases}\]Note that $\supp{p}\cap\supp{\pi(p')}=\varnothing$, hence $p||\pi(p')$.\end{proof}

\noindent Since $\sa_{\w_1}$ is weakly homogeneus and $\sa_{\omega_1 \setminus \xi}
\cong \sa_{\omega_1}$, \eqref{eq2} gives us 
\[\one\dststile{\LL[\res{\,g}{\,\xi}]}{\sa_{\w_1}} \check{p}\dststile{\LL[\check{\res{\,g}{\,\xi}}][\dot{g}]}{\PP_B}\text{``$\varphi(\cdot\;, \cdot\;, \check{\check{\vec{x}}}, \check{\check{\vec{\alpha}}}, \{({\check y},q): \theta(y,q)\})$ defines a well-ordering of ${}^\w2$."}\] 
%where $\dot{b}$ is considered as the set $\{v: \theta(v)\}$ where \centm{$\theta(v)\iff v\in\PP_B \text{ and $v\in h$}$}with $h$ being a $\PP_B$-generic filter.  By \eqref{eq1}, $\theta$ is a $\bsi{3}$-formula so $\dot{b}$ is a $\bsi{3}$-set. 
 
 \noindent Let $g^*$ be $\sa_{\w_1}$-generic over $\LL[g]$ so that $\res{g}{[\xi,\w_1)}$ and $g^*$ are (or may be construed as) mutually $\sa_{\w_1}$-generics over $L[\res{g}{[\xi, \w_1)}]$, and let $h^*$ be $\PP_B$-generic over $\LL[\res{g}{\xi}, g^*]$ with $p\in h^*$.  We have that
\centm{$\LL[\res{g}{\xi}, g^*][h^*]\models \text{``$\varphi(\cdot\;, \cdot\;, \vec{x}, \vec{\alpha}, b^*)$ defines a well-ordering of ${}^\w2$,''}$}where $b^*:=\bigcup h^*$ is the Burstin basis added by $h^*$.  Since \centm{$\RR\cap\LL[\res{g}{\xi}, g^*][h^*]=\RR\cap\LL[\res{g}{\xi}, g^*]\neq\RR\cap\LL[g]=\RR\cap\LL[g][h]$}
we can find some $\beta$, some $n<\w$, and $i \in \{ 0, 1 \}$ such that
{\begin{enumerate-(i)}\item $\LL[g,h]\models \text{``the $n^{th}$ digit of the $\beta^{th}$ element of ${}^\w2$ given by $\varphi(\cdot\;, \cdot\;, \vec{x}, \vec{\alpha}, b)$ is $i$"}$
\item $\LL[\res{g}{\xi},g^*][h^*]\models \text{``the $n^{th}$ digit of the $\beta^{th}$ element of ${}^\w2$ given by $\varphi(\cdot\;, \cdot\;, \vec{x}, \vec{\alpha}, b^*)$ is $1-i$"}$\end{enumerate-(i)}} 

\noindent Thus there exist two conditions  $p_0\in h$ and $p_1\in h^*$  below $p$ such that
{\begin{enumerate-(is)}\item $p_0\dststile{\LL[g]}{\PP_B}\text{``the $\check{n}^{th}$ digit of the $\check{\beta}^{th}$ element of ${}^\w2$ given}\\\hspace{5cm} 
\text{ by $\varphi(\cdot\;, \cdot\;, \check{\vec{x}}, \check{\vec{\alpha}}, \{({\check y},q): \theta(y,q)\})$ is $\check{i}$''}$
\item $p_1\dststile{\LL[\res{g}{\,\xi}, \,g^*]}{\PP_B}\text{``the $\check{n}^{th}$ digit of the $\check{\beta}^{th}$ element of ${}^\w2$ given}\\\hspace{5cm} 
\text{ by $\varphi(\cdot\;, \cdot\;, \check{\vec{x}}, \check{\vec{\alpha}}, \{({\check y},q): \theta(y,q)\})$ is $\check{1-i}$''}$  
\end{enumerate-(is)}}
\noindent Pick $\zeta\geq\xi$, $\zeta<\w_1$ such that $p_0\in\LL[\res{g}{\zeta}]$ and $p_1\in\LL[\res{g}{\xi}, \res{g^*}{\zeta}]$, say $\xi+\zeta=\zeta$.  Then (i)* and (ii)* above give us
{\[(*)\begin{cases}\one\dststile{\LL[\res{g\,}{\,\zeta}]}{\sa_{\w_1}}  \:\check{p_0}\dststile{\LL[\res{g\,}{\,\zeta}][\dot{g}]}{\PP_B}\text{``the $\check{\check{n}}^{th}$ digit of the $\check{\beta}^{th}$ element of ${}^\w2$ given by}\\\hspace{5cm} 
\text{ $\varphi(\cdot\;, \cdot\;, \check{\check{\vec{x}}}, \check{\check{\vec{\alpha}}}, \{({\check y},q): \theta(y,q)\})$ is $\check{\check{i}}$''}\\
\one\dststile{\LL[\res{g\,}{\,\xi}, \res{g^*}{\,\zeta}]}{\sa_{\w_1}} \:\check{p_1}\dststile{\LL[\res{g\,}{\,\xi},\res{g^*}{\,\zeta}][\dot{g}]}{\PP_B}\text{``the $\check{\check{n}}^{th}$ digit of the $\check{\beta}^{th}$ element of ${}^\w2$ given by}\\\hspace{5cm} 
\text{$\varphi(\cdot\;, \cdot\;, \check{\check{\vec{x}}}, \check{\check{\vec{\alpha}}}, \{({\check y},q): \theta(y,q)\})$ is $\check{\check{1-i}}$''}
\end{cases}\]}
\noindent Now  we want to make sure that the conditions $p_0$ and $p_1\in L[g,g^*]$ are compatible.

\begin{claim}\label{key} $p_0\cup p_1$ is linearly independent.\end{claim}

\begin{proof} We may assume without loss of generality that \centm{$\LL[\res{g}{\xi}]\models ``p_0 \text{ is a Burstin basis."}$}In particular, it is
true in $\LL[\res{g}{\xi}]$ that $p_0$ is a Hamel basis.\:  Suppose that there are $\vec{y}\in p$, $\vec{y}_0\in p_0\smallsetminus p$, $\vec{y}_1\in p_1\smallsetminus p$ and some vectors of rational numbers $\vec{q}, \vec{q_0}$, $\vec{q_1}$ such that
\begin{equation}\label{eq15}\vec{q}\cdot\vec{y}+ \vec{q_0}\cdot\vec{y_0}+\vec{q_1}\cdot \vec{y_1}=0\end{equation}
\noindent By mutual genericity we have \centm{$\vec{q}\cdot\vec{y}+\vec{q_0}\cdot\vec{y_0}= -\vec{q_1}\cdot\vec{y_1}\in \LL[\res{g}{\zeta}]\cap\LL[\res{g}{\xi},\res{g^*}{\zeta}]=\LL[\res{g}{\xi}]$} 
Since $p$ is a Hamel basis for the reals of $\LL[\res{g}{\xi}]$, there exists some $\vec{z}_1\in[p]^{<\w}$, $\vec{r}_1\in[\QQ]^{<\w}$ such that 
\vspace{-5mm}\centm{$\vec{r}_1\cdot\vec{z}_1=-\vec{q}_1\cdot\vec{y}_1$}  Since $p_1\supset p$ is linearly independent it follows that $\vec{r}_1=0=\vec{q}_1$.   Coming back to the equation \eqref{eq15}, we now have that 
\vspace{-5mm}\centm{$\vec{q}\cdot \vec{y}+\vec{q}_0\cdot\vec{y}_0=0$} Since  $q_0\supset p$ is also linearly independent, we conclude in that $\vec{q}=0=\vec{q}_0$.  Hence  $p_0\cup p_1$ is linearly independent.
\end{proof}

\noindent We may construe ${\res{g}{[\zeta, \w_1)}}^\smallfrown g^*$ as $\sa_{\w_1}$-generic over $\LL[\res{g}{\xi}, \res{g^*}{\zeta}]$ as well as over $L[\res{g}{\zeta}]$.  Therefore by $(*)$ it follows that
{\[(**)\begin{cases}p_0\dststile{\LL[g][g^*]}{\PP_B}\text{``the $\check{n}^{th}$ digit of the $\check{\beta}^{th}$ element of ${}^\w2$ given by}\\\hspace{5cm} 
\text{ $\varphi(\cdot\;, \cdot\;, \check{\vec{x}}, \check{\vec{\alpha}}, \{({\check y},q): \theta(y,q)\})$ is $\check{i}$"}\\
p_1\dststile{\LL[g][g^*]}{\PP_B}\text{``the $\check{n}^{th}$ digit of the $\check{\beta}^{th}$ element of ${}^\w2$ given by}\\\hspace{5cm} 
\text{ $\varphi(\cdot\;, \cdot\;, \check{\vec{x}}, \check{\vec{\alpha}}, \{({\check y},q): \theta(y,q)\})$ is $\check{1-i}$"}  
\end{cases}\]}
By claim \ref{key} and lemma \ref{5}, there is some $q\leq p_0, p_1$, $q\in{\PP_B}^{\LL[g,\, g^*]}$.  But then, $q$ forces the contradictory statements
from the matrices of  $(**)$.   This concludes the proof.
\end{proof}

The previous proof in fact shows the following.
\begin{lemma}\label{lemma_6.2} Let $g$ be $\sa_{\w_1}$-generic over $\LL$, let $h$ be $\PP_B$-generic $h$ over $\LL[g]$ and let $b = \bigcup h$ be the Burstin basis added by $h$. 
Inside $L[g,h]$, there are Turing-cofinally many $x \in {\mathbb R}$
such that if $X \subset L[x]$, $X \in {\rm OD}_{x,b}$, then $X \in L[x]$.
\end{lemma}

By standard arguments, Lemma \ref{lemma_6.2} then implies.
\begin{lemma}\label{lemma_6.3}
Let $g$ be $\sa_{\w_1}$-generic over $\LL$, let $h$ be $\PP_B$-generic $h$ over $\LL[g]$ and let $b = \bigcup h$ be the Burstin basis added by $h$.  Let 
$W= \LL(\RR, b)^{\LL[g,h]}$.
Then \[{}^\omega W \cap \LL[g,h] \subset W.\] In particular, $W$ is a model of 
$\DC$, the principle of dependent choice.
\end{lemma}

%\begin{remark}The argument given above shows that the set \centm{$\{p\in\PP_B^{L[g]}:\forall \varphi\forall\vec{u}\in L[p](p\text{ decides }\varphi(\vec{u}, \dot{b}))\}$} is dense.\end{remark}

\begin{theorem}\label{summary}Let $g$ be $\sa_{\w_1}$-generic over $\LL$, and let $b$ be $\PP_B$ generic over $\LL[g]$.  Let 
\centm{$W= \LL(\RR, b)^{\LL[g,\,b]}$.}Then, $W\models \ZF+\DC$ and in $W$ there are Luzin, Sierpi\'{n}ski, Vitali and a Burstin basis  but in $W$ there is no a well-ordering of $\RR$.  \end{theorem}
\begin{proof}Clearly Lemma \ref{lemma_6.3} gives $W\models \ZF+\DC$.  Now, as $\PP_B$ is $\w$-closed, $\RR\cap W=\RR\cap \LL[g]$, so that $W\models \text{``$b$ is a Burstin basis"}$.  This means that in $W$, we have a Bernstein set and a Hamel basis.   Hence, in view of \ref{vi}, there is a Vitali in $W$ set induced by $b$.  By Corollary \ref{4}, $W$ has a  Luzin as well as a Sierpi\'{n}ski set.    Finally, by \ref{m2}, in $W$ there is no a well-ordering of the reals, as required.\end{proof}

\section{Further remarks: ultrafilters on $\omega$, Mazurkiewicz sets, etc.}

Let $g$ be $\sa_{\omega_1}$-generic over $L$. 

By \cite[Theorem 6]{laver}, in $L[g]$
there is an ultrafilter on $\omega$ which is generated by an ultrafilter in $L$. In fact,
if $U \in L$ is a selective ultrafilter on $\omega$, then $U$ generates an ultrafilter in
$L[g]$
(see \cite{olga}). This implies that the model $W = \LL(\RR, b)^{\LL[g,\,b]}$
from Theorem \ref{summary} has ultrafilters on $\omega$. 

A set $M\subseteq \RR^2$ is a \emph{Mazurkiewicz}\index{special sets!Mazurkiewicz} set if $M$ intersects every straight line in exactly two points. Masurkiewicz showed in
$\ZFC$ that Mazurkiewicz sets exist, see \cite{mazur}.
We may force with a poset ${\mathbb P}_M$ consisting of
``local'' Mazurkiewicz sets over $L[g]$ in much the same way as Definition
\ref{defn_BF} gave a forcing whose conditions
are ``local'' Burstin bases. If $m$ is the set added by ${\mathbb P}_M$,
then $m$ will be a Mazurkiewicz set in $\LL(\RR, m)^{\LL[g,\,m]}$ and this
model will not have a well-ordering of the reals. This result is produced in \cite{maz}.

We may in fact force with the product ${\mathbb P}_B \times {\mathbb P}_M$ over
$L[g]$ and get a model with a Burstin base and a Mazurkiewicz set with no well-order of the reals.

In the same fashion, one may add further ``maximal independent'' sets generically over
$L[g]$, e.g.\ selectors for $\Sigma^1_2$ definable equivalence relations, without
adding a well-ordering of ${\mathbb R}$. (Cf.\ \cite{budinas}.)

\bibliography{specialsets}
  
\end{document}